\documentclass[11pt,psamsfonts]{amsart}

\usepackage[left=1.0in,right=1.0in,top=1.0in,bottom=1.0in]{geometry}

\def\R{{\mathbb R}}

\def\T{{\mathbb T}}

\def\Q{{\mathbb Q}}
\def\Z{{\mathbb Z}}
\def\N{{\mathbb N}}

\def\cont{\mathfrak{c}}
\def\Ni{\mathsf{N}}

\def\HMi{\mathsf{HM}}

\def\vNk{von Neumann kernel}
\def\pvN{potential von Neumann kernel}
\def\Prm{{\mathbb P}}

\def\map{{minimally almost periodic}}
\def\Zar{\mathfrak Z}

\def\axiom#1{\hbox{$\nabla_\kappa$}}
\def\Axiom#1{\hbox{$\nabla_\kappa'$}}

\usepackage{amsmath}
\usepackage{amssymb}

\usepackage{mathrsfs}

\newtheorem{thm}{Theorem}[section]
\newtheorem{theorem}[thm]{Theorem}

\newtheorem{corollary}[thm]{Corollary}

\newtheorem{proposition}[thm]{Proposition}
\newtheorem{lemma}[thm]{Lemma}

\theoremstyle{definition}

\newtheorem{definition}[thm]{Definition}
\newtheorem{example}[thm]{Example}

\newtheorem{remark}[thm]{Remark}

\newtheorem{problem}[thm]{Problem}

\def\HM{Hartman-Mycielski}

\title[]
{Final solution of Protasov-Comfort's problem on minimally almost periodic group topologies}
\author[]
{Dikran Dikranjan}
\address[D. Dikranjan]{\hfill\break
Dipartimento di Matematica e Informatica
\hfill\break
Universit\`{a} di Udine
\hfill\break
Via delle Scienze  206, 33100 Udine
\hfill\break
Italy}
\email{dikranja@dimi.uniud.it} 
\author[]
{Dmitri Shakhmatov}
\address[D. Shakhmatov]{\hfill\break
Department of Mathematical Sciences, Faculty of Science\hfill\break
Ehime University, Matsuyama 790--8577\hfill\break Japan}
\email{dmitri.shakhmatov@ehime-u.ac.jp}
\thanks{The first named author gratefully acknowledges the FY2013 Long-term visitor grant~L13710 by the Japan Society for the Promotion of Science (JSPS)}
\thanks{The second named author was partially supported by the Grant-in-Aid for
Scientific Research~(C) No.~26400091 by the Japan Society for the Promotion of Science (JSPS)}
\begin{document}
\begin{abstract}
We prove that an abelian group admits a minimally almost periodic (MinAP) group topology if and only if  it is connected in its Markov-Zariski topology. In particular, every unbounded 
abelian group admits a MinAP group topology. This answers positively a question set by Comfort \cite{Co}, as well as several weaker forms proposed recently by Gabriyelyan \cite{G3}. Using this characterization we answer also two open questions of Gould \cite{G}. We prove that 
a subgroup $H$ of an abelian group $G$ can be realized as the von Neumann kernel of $G$ equipped with some Hausdorff group topology if and only if $H$ is contained in the connected component of zero of $G$ with respect to its Markov-Zariski topology. 
This completely resolves a question of Gabriyelyan \cite{G0}, as well as some of its particular versions which were open.
\end{abstract}
\maketitle

As usual, $\N$ denotes the set of natural numbers, $\Z$ denotes the group of integers, $\Prm$ denotes the set of all prime numbers and $\cont$ denotes the cardinality of the continuum.

For a group $G$ and $m\in \N$, we let 
$$
mG=\{mg:g\in G\}\  \ \mbox{ and }\  \ G[m]=\{g\in G: mg = e\},
$$ 
where $e$ is the identity element of $G$. When $G$ is abelian, we use $0$ instead of $e$. A group $G$ has {\em finite exponent\/} if $mG=\{e\}$ for some integer $m\ge 1$.

Recall that a non-trivial abelian group $G$ of finite exponent is a direct sum of cyclic groups 
\begin{equation}
\label{eq:1}
G = \bigoplus _{p \in \pi(G)} \bigoplus_{i=1}^{m_p} \Z(p^i)^{(\alpha_{p,i})},
\end{equation}
where $\pi(G)$ is a non-empty finite set of primes and the cardinals $\alpha_{p,i}$ are known as {\em Ulm-Kaplanski invariants\/} of $G$. Note that while some of them may be equal to zero, the cardinals $\alpha_{p,m_p}$ are supposed to be positive; they are called the {\em leading Ulm-Kaplanski invariants\/} of $G$.  

\section{Protasov-Comfort problem on minimally almost periodic group topologies}

Our main results provide a solution to Protasov--Comfort problem on minimally almost periodic group topologies on the abelian groups. 
 For the sake of completeness, we recall the relevant pair of notions, due to von Neumann \cite{vN}, in the abelian case. 
 A {\em character} of a topological abelian group $G$ is a continuous homomorphism  $G \to \T$. A topological abelian group $G$ is called 

\begin{itemize}
  \item[(a)] {\em minimally almost periodic}, if every character $G \to \T$ is trivial;  
  \item[(b)] {\em maximally almost periodic}, if the characters $G \to \T$ separate the points of $G$. 
\end{itemize}

It is easy to see that every abelian group admits a maximally almost periodic group topology (e.g., the discrete one). 
Examples of minimally almost periodic abelian groups are not easy to come by. The first known examples were those borrowed from Analysis, namely vector topological spaces without non-trivial continuous functionals (more precisely, the underlying topological abelian group of such a space turns out to be minimally almost periodic \cite[Section 23.32]{HR}, example of vector spaces with this property can be found in \cite{Day}). 
Nienhuys \cite{N} built a solenoidal (so, of size at least $\cont$) and monothetic minimally almost periodic group, which yields the existence of a minimally almost periodic group topology also on $\Z$ (see Section \ref{Nienhyus:section} for further details). All these examples are connected. 

To the best  of our knowledge, the first explicit and also quite simple example of a countable \map\ group was given by Prodanov \cite{Pr}. Later,  Ajtai, Havas and  Koml\' os \cite{AHK} provided minimally almost periodic group topologies on $\Z$ and some countably infinite direct sums of simple cyclic groups, deducing that every abelian group admits a group topology that is not maximally almost periodic.
In his review to their paper Protasov \cite{P} posed the natural question of whether {\em every} infinite abelian group admits a minimally almost periodic group topology.

 In September 1989 Remus noticed the following counter-example to this question.

\begin{example}\label{Remus:example} (Remus)
{\em The group $ G = \Z(2)\times \Z(3)^\omega$ does not  admit any \map\ group topology\/}
 since the projection $p: \Z(2)\times \Z(3)^\omega\to \Z(2)$ is
a continuous character for every Hausdorff group topology on $G$.
Indeed, $\ker p = \{x\in G: 3x = 0\}$ is closed (and hence, also open) in every Hausdorff group topology on $G$.
\end{example}
Motived by this example, Comfort \cite[Question 3J.1]{Co} modified the original Protasov's question to the following 

\begin{problem}\cite[Question 521]{Co}\label{Q:C}
 Does every Abelian group which is not of bounded order admit a minimally almost periodic topological group topology? What about the countable case?
\end{problem}

Zelenyuk and Protasov \cite{ZP} introduced a new technique for building minimally almost periodic group topologies 
on countable groups, using $T$-sequences. Applying this method, Gabriyelyan \cite{G3} obtained a description of the bounded abelian groups
admitting a minimally almost periodic group topology, by essentially showing that these are precisely the bounded groups whose leading Ulm-Kaplansky invariants are infinite. 

Gabriyelyan \cite[Theorem 2]{G1} also  proved that all finitely generated infinite abelian groups admit a minimally almost periodic group topology, and he recently extended this to all countably infinite unbounded groups \cite{G3}, resolving the second part of Question \ref{Q:C} (a short and self-contained proof of this result is given in Lemma \ref{Saak}).  Nevertheless, Problem \ref{Q:C} remained open for all uncountable groups. 

\section{Necessary conditions for the existence of a minimal almost periodic group topology}

\begin{proposition}\label{Necessary:condition} If $G$ is a minimally almost periodic abelian group, then every continuous homomorphism  $f:G\to K$ of $G$ into a compact group $K$ is trivial. In particular, every proper closed subgroup $H$ of $G$ has infinite index in $G$.
\end{proposition}

Markov \cite{Mar1,Mar} says that subset $X$ of a group $G$ is {\em unconditionally closed\/} in $G$ if $X$ is closed in every Hausdorff group topology on $G$.

\begin{corollary}
\label{cor:necessary:condition}
If an abelian group $G$ admits a  minimally almost periodic abelian group topology, then every  proper unconditionally closed subgroup $H$ of $G$ has infinite index in $G$.
\end{corollary}

This corollary explains why the group $G$ from Example \ref{Remus:example} does not admit a minimally almost periodic group topology. Indeed, the subgroup $G[3]$ is unconditionally closed in $G$ and has index $2$.

It is natural to ask if the converse to Corollary \ref{cor:necessary:condition} also holds.

\begin{problem}
\label{Main:question} 
If all proper unconditionally closed subgroups of a group $G$ have infinite index, does then $G$ admit a  minimally almost periodic group topology? 
\end{problem}

Due to the fundamental fact that all unconditionally closed subsets of abelian groups are algebraic \cite{DS-relection, DS_JA}, these sets are precisely the closed sets of the Zariski topology $\Zar_G$ of $G$ \cite{DS_JA}, having as closed sets the algebraic sets of $G$. Following \cite{DT}, call a group $G$ {\em Zariski-connected} (briefly, {\em $\Zar$-connected}) if $(G,\Zar_G)$ is connected. 

Since the notion of an unconditionally closed subset of a group $G$ involves checking closedness of this set in {\em every\/} Hausdorff group topology on $G$, in practice it is hard to decide if a given subgroup of $G$ is unconditionally closed in $G$ or not. Our next proposition provides two equivalent conditions.

\begin{proposition}
\label{reformulation:of:Markov:condition} For an abelian group $G$, the following conditions are equivalent: 
\begin{itemize}
\item[(a)] $G$ is {$\Zar$-connected}; 
\item[(b)] all proper unconditionally closed subgroups of $G$ have infinite index;
\item[(c)] for every $m\in\N$, either $mG =\{0\}$ or  $|mG| \ge \omega$. 
\end{itemize}
\end{proposition}

\begin{proof}
According to \cite{DS-relection, DS_JA}, a proper unconditionally closed subgroup $H$ of $G$ has the form $H = G[m]$ for some $m>0$. Since 
$G/G[m]\cong mG$, the index of $H$ in $G$ coincides with $|G/H| = |G/G[m]|= |mG|$. This proves the equivalence of (b) and (c).
The equivalence of (a) and (b) is established in \cite[Theorem 4.6]{DS_JA}.  
\end{proof}

If an abelian group $G$ is not bounded torsion, then 
$G$ satisfies item (c) of Proposition \ref{reformulation:of:Markov:condition}, applying which we conclude that all proper unconditionally closed subgroups of $G$ have infinite index.
This shows that Problem \ref{Q:C} is a particular case of the more general Problem \ref{Main:question}.

The equivalence of items (a) and (b) in Proposition \ref{reformulation:of:Markov:condition} allows one to reformulate Problem \ref{Main:question} is follows: {\em Does every abelian $\Zar$-connected group admit a minimally almost periodic group topology?\/}

Based on Proposition \ref{reformulation:of:Markov:condition}, one can re-formulate the property of $\Zar$-connectedness for abelian groups of finite exponent
in terms of their Ulm-Kaplanski invariants: 

\begin{proposition}
\label{Kirku*}
A non-trivial abelian group $G$ of finite exponent  satisfies one (and then all) of the equivalent conditions of 
Proposition \ref{reformulation:of:Markov:condition} if and only if all leading Ulm-Kaplanski invariants of $G$ are infinite. 
\end{proposition}

\begin{proof} A proof of this statement can be found in \cite{CD} (more precisely, this is the equivalence (a$_2$) and  (a$_3$) of \cite[Lemma 2.13]{CD}). For the sake of reader's convenience we include a proof here. 

Write the group $G$ as in \eqref{eq:1}, and let $k = \prod_{q\in \pi(G)} q^{m_q}$ be the exponent of $G$. In order to compute the leading Ulm-Kaplanski invariant $\alpha_{p,m_p}$ for $p\in \pi(G)$, let 
$$
k_p = \frac{k}{p}= p^{m_p-1}\cdot \prod_{q\in \pi(G)\setminus \{p\}} q^{m_q}.
$$ 
Then $k_pG \cong \Z(p)^{(\alpha_{p,m_p})}$, so  
\begin{equation}
\label{eq:2}
1< |k_pG|=\begin{cases} p^{\alpha_{p,m_p}}\ \mbox{ if $\alpha_{p,m_p}$ is finite}\\
\alpha_{p,m_p}
\ \mbox{ if $\alpha_{p,m_p}$
is infinite}.
\end{cases}
\end{equation}

If $G$ satisfies item (c) of Proposition \ref{reformulation:of:Markov:condition}, then $|k_pG | \ge \omega$ by \eqref{eq:2}, so $\alpha_{p,m_p}\ge \omega$ as well. 

Now suppose that all leading Ulm-Kaplanski invariants of $G$ are infinite.
Fix  $m\in \N$ with $|mG| > 1$. There exists at least one $p \in \pi(G)$, such that $p^{m_p}$ does not divide $m$.  Let $d$ be the greatest common divisor of $m$ and $k$. Then $mG = dG$, hence from now on we can assume without loss of generality that $m=d$ divides $k$.  As $p^{m_p}$ does not divide $m$, it follows that $m$ divides $k_p$.  Therefore, $k_p G$ is a subgroup of $mG$, and so $|mG| \geq |k_p G| = \alpha_{p,m_p} \geq \omega$ by \eqref{eq:2}. 
This shows that, for every $m\in\N$, either $mG =\{0\}$ or  $|mG| \ge \omega$. Therefore, $G$ satisfies item (c) of Proposition \ref{reformulation:of:Markov:condition}.
\end{proof}

\section{Characterization of abelian groups admitting a \map\ group topology}
\label{sec:main:results}

Our first theorem provides a final solution of Problem \ref{Q:C}.

\begin{theorem}\label{CH:theorem} Every unbounded abelian group admits a \map \ group topology.
\end{theorem}

The proof of Theorem \ref{CH:theorem} is postponed until Section \ref{proofs}.

The complementary bounded case is covered by the next theorem.

\begin{theorem} \label{new:corollary1}
For a bounded abelian group $G$ the following are equivalent: 
\begin{itemize}
\item[(a)]  $G$ admit a \map \ group topology; 
\item[(b)]  $G$ is {$\Zar$-connected}; 
\item[(c)] all leading Ulm-Kaplanskly invariants of $G$ are infinite.
\end{itemize}
\end{theorem} 

The 
equivalence of (a) and (c)
characterization of bounded abelian groups admitting a \map\ group topology 
has been  recently obtained by 
Gabriyelyan \cite[Corollary 1.4]{G3}. 
Since our  proof of Theorem \ref{new:corollary1}  based on the same ideas as that of Theorem  \ref{CH:theorem} is also  much shorter and simpler than Gabriyelyan's proof, we decided to provide it in Section \ref{proofs0} for the sake of completeness.

Our third theorem unifies Theorems
\ref{CH:theorem}
and
\ref{new:corollary1}, thereby providing a final solution to the general Problem \ref{Main:question}.

\begin{theorem}\label{main:theorem}
For an abelian group an abelian group $G$, the following are equivalent: 
\begin{itemize}
\item[(a)] $G$ admits a \map \ group topology; 
\item[(b)] $G$ is {$\Zar$-connected}; 
\item[(c)] all proper unconditionally closed subgroups of $G$ have infinite index;  
\item[(d)] for every $m\in\N$, either $mG =\{0\}$ or  $|mG| \ge \omega$. 
\end{itemize}
\end{theorem}
\begin{proof}
Indeed, if $G$ is unbounded, then Theorem \ref{CH:theorem} and Proposition \ref{reformulation:of:Markov:condition} apply. If $G$ is bounded, then
Proposition \ref{reformulation:of:Markov:condition}, Proposition \ref{Kirku*} and  Theorem \ref{new:corollary1} apply. 
\end{proof}

Theorem \ref{main:theorem} answers positively also the following  weaker versions of  Problem \ref{Main:question} that  also remained open:

\begin{itemize}
\item[(a)]  \cite[Question 3.10]{G3} If an unbounded abelian group $G$ admits a compact group topology, does $G$ admit also a \map \ group topology? 
\footnote{The group $ G = \Z(2)\times \Z(3)^\omega$ from Example \ref{Remus:example} obviously carries a compact group topology. This explains the restriction of unboundedness of $G$ in item (a).} What about the groups $\prod_p\Z(p)$, or $\Delta_p$, where $\Delta_p$ is the group of $p$-adic integers ? 
\item[(b)] \cite[Question 4.1]{G3} Let $G$ be an uncountable abelian group of infinite exponent (for example, $G$ is an uncountable torsion-free abelian group). Does $G$  
admit a \map \  group topology? 
\item[(c)] \cite[Question 4.5]{G3} Does every uncountable indecomposable Abelian group admit
a \map \  group topology?
\end{itemize}

\medskip
Finally, Theorem  \ref{CH:theorem}  answers also the following question from  \cite{G0,G3} generalizing Problem \ref{Q:C}: 

\begin{itemize}
\item[(q)] \cite[Question 4.4]{G3} (also \cite[Problem 5]{G0}) 
Describe all infinite abelian groups $G$ that admit a \map \ group topology.
\end{itemize}

\medskip

In \cite{G}, the following two questions were raised. Let $H$ be a normal subgroup of a group $G$. 
\begin{itemize}
\item[(1)] If we are given minimally almost periodic topologies for $H$ and $G/H$, is there always a minimally almost periodic topology for $G$ which generates those topologies on $H$ and $G/H$?
\item[(2)] If both $H$ and $G/H$ admit minimally almost periodic topologies, must $G$ admit a minimally almost periodic topology?
\end{itemize}

Our last theorem negatively answers both questions. Indeed, item (c) of the next theorem provides a stronger negative answer to (1), while 
the conjunction of items (a) and (b) of this theorem obviously answers item (2) in the negative. 

\begin{theorem}\label{Frank}
There exists  a countable abelian group $G$ with a subgroup $H$ having the following properties: 
\begin{itemize}
\item[(a)]  $G$ does not admit a \map\ group topology;
\item[(b)]  both $H$ and $G/H\cong H$ admit \map\ group topologies; 
\item[(c)]  there does not exist any group topology $\tau$ on $G$ such that both $H$ and $G/H$, equipped with the induced and the quotient topology of $\tau$ respectively, can be simultaneously minimally almost periodic. 
\end{itemize}
\end{theorem}

\begin{proof}
Let  $V = \Z(2)^{(\omega)}\oplus \Z(4)$, $G = \Z(2)^{(\omega)}\oplus V$ and $H = V[2]$.

(a) Since  $G[2]=V[2]$  has index $2$ in $G$, it follows from  Proposition \ref{Necessary:condition}
that $G$ admits no \map \ group topologies.
 
(b) Obviously, $H \cong G/H \cong\Z(2)^{(\omega)}$.  Applying Theorem \ref{new:corollary1}, we conclude that $H \cong G/H$ admits a \map\ group topology.

(c) Assume that for some group topology $\tau$ on $G$ both $H$ and $G/H$ (equipped with the induced and the quotient topology of $\tau$, respectively)  are 
minimally almost periodic.  Then $\tau$ itself is minimally almost periodic.
 (A proof can be found, for example, in \cite[Theorem 2.2.10]{G}.) \footnote{For reader's convenience we propose
a short argument here. If $\chi: G\to \T$ is a character, then $\chi\restriction_H$ must be trivial by hypothesis. Therefore, there exists a character $\xi : G/H \to \T$ such that $\chi = \xi \circ q$, where $q: G\to G/H$ is the canonical homomorphism.  As $\xi$ is trivial again by  hypothesis, we conclude that $\chi$ is trivial.
This proves that $(G,\tau)$ is \map.} This contradicts item (a).
\end{proof}

\section{The realization problem of von Neumann's kernel}

For a topological abelian group $G$, the {\em von Neumann kernel} $n(G)$ of $G$ is the subgroup of all points of $G$ where each character of $G$ vanish.
Clearly,  $G$ is minimally almost-periodic (maximally almost-periodic) precisely when $n(G) = G$ ($n(G) = \{0\}$, respectively). 

\begin{definition}\label{Definition:potentially:vN}
Let $H$ be a subgroup of an abelian group $G$. We say that $H$ is a {\em potential von Neumann kernel} of $G$, if there exists a Hausdorff group topology $\tau$ on $G$ such that $n(G,\tau) = H$. 
\end{definition} 

 The following problem, raised in \cite{G0}, is a generalization of the problem of finding a \map\ group topology on an abelian group considered in Section \ref{sec:main:results}; indeed, an abelian group $G$ admits a \map\ group topology if and only if $G$ is \pvN\ of itself.

\begin{problem}\label{Gen:Prob}\cite{G0}
Describe all potential von Neuman kernels of a given abelian group $G$.  
\end{problem} 

Gabriyelyan 
resolved this problem for bounded abelian groups.
 
 \begin{theorem}\label{corollary2} \cite[Theorem 1.2]{G3} 
Let $G$ be a bounded abelian group and let $H$ be a subgroup of $G$. Then $H$ is a \pvN \ of $G$ if and only if 
$G$ contains $\Z(k)^{(\omega)}$, where $k = exp(H)$ is the exponent of $H$. 
 \end{theorem} 

Furthermore, 
Gabriyelyan 
resolved the following particular case of Problem \ref{Gen:Prob}.
\begin{theorem}
\cite[Theorem 1.3]{G3} 
\label{Gab-2}
Every
bounded subgroup of an unbounded abelian group $G$  is a \pvN \ of $G$.
\end{theorem}

The following easy lemma 
is
helpful 
for
finding 
a 
necessary condition that all \pvN s must satisfy. 

\begin{lemma}\label{very:easy:lemma} The \vNk\  of a topological group $G$ is contained in every open subgroup of $G$ and contains every \map\ subgroup of $G$.  
\end{lemma} 

\begin{proof}
If $H$ is an open subgroup of $G$, then $G/H$ is discrete, so it is maximally almost periodic. Since the characters of $G/H$ separate points of $G/H$, 
we get $n(G)\subseteq H$. The second assertion is obvious. 
\end{proof}

\begin{lemma}\label{lemma:necessary:condition} All \pvN s of an abelian group $G$ are contained in the intersection of all unconditionally closed subgroups of $G$ of finite index.
\end{lemma} 

\begin{proof} 
Let $N$ be an unconditionally closed subgroup of finite index  of  $G$ and let $H$ be a \pvN. To see that 
$N \subseteq H$ equip $G$ with the group topology $\tau$ witnessing $H = n(G,\tau)$. Then $N$ is $\tau$-clopen,
so $H = n(G,\tau)\subseteq N$ by Lemma \ref{very:easy:lemma}. 
\end{proof}

The above lemma makes it important to compute the intersection of all unconditionally closed subgroups of finite index  of an abelian group $G$. 
To this end we need the following definition.

\begin{definition}\label{lemma2}\cite[Definition 4.3]{DS_JA}
Let $G$ be an abelian group.
\begin{itemize}
\item[(i)] If $G$ is bounded, then the {\em essential order} $eo(G)$ of $G$ is the smallest positive integer $n$ such that $nG$ is finite.
\item[(ii)] If $G$ is unbounded, we define $eo(G) = 0$.
\end{itemize}
\end{definition}

The notion of the essential order of a bounded abelian group $G$, as well as the notation $eo(G)$, are due to Givens and Kunen \cite{GK}, although the definition in \cite{GK} is different (but equivalent) to this one.

\begin{theorem}\label{Theorem:JA}\cite[Theorem 4.6]{DS_JA}
Let $G$ be an abelian group and $n = eo(G)$. Then:

(i) $G[n]$ is a $\Zar_G$-clopen subgroup of $G$,

(ii) $G[n]$ coincides with the connected component $c_\Zar(G)$ of $0$ in $(G,\Zar_G)$.
 \end{theorem} 

By the definition of $eo(G)$, the subgroup $G[n]$ has finite index, and this is the smallest subgroup of the form $G[m]$ that has finite index. 
So $c_\Zar(G) = G[n]$ coincides with the intersection of all unconditionally closed subgroups of $G$ of finite index. 
Therefore,
from  Lemma \ref{lemma:necessary:condition} and Theorem \ref{Theorem:JA} one obtains the following 
necessary condition for a subgroup to be a \pvN. 

\begin{corollary}\label{corollary3*} 
Every
\pvN \ $H$ of an abelian group $G$
satisfies $H\subseteq c_\Zar(G)$.
 \end{corollary}

 In the next lemma we collect several equivalent forms of the necessary condition $H\subseteq  c_\Zar(G)$ in the case when $G$ is bounded.

\begin{lemma}\label{lemma3} Let $G$ be a bounded abelian group and $m = eo(G)$. For every subgroup $H$ of $G$ TFAE: 

(a) $H\subseteq  c_\Zar(G) \ (=G[m]) $; 

(b) $exp(H) | m$; or equivalently, $mH = 0$;

(c) $G$ contains $\Z(k)^{(\omega)}$, where $k = exp(H)$. 
 \end{lemma} 

\begin{proof}
(a) and (b) are obviously equivalent. By \cite[Proposition 4.12]{DS_JA}, (b) is equivalent to (c). 
\end{proof}

Our main theorem in this section shows that this necessary condition from Corollary \ref{corollary3*} is also sufficient for the realization of the \vNk.
 
\begin{theorem}\label{Main:Conjecture} A subgroup $H$ 
of  an abelian group  $G$ is a \pvN \ of $G$ if and only if $H\subseteq  c_\Zar(G)$. 
 \end{theorem} 
\begin{proof}
The necessity was proved in Corollary \ref{corollary3*}. 
To prove the sufficiency, assume that $H\subseteq  c_\Zar(G)$ and consider two cases.

\smallskip
{\em Case 1\/}. {\sl $H$ is bounded.\/}
If $G$ is unbounded, 
then 
$H$ is a potential \pvN\
by Theorem \ref{Gab-2}.

Suppose now that $G$ itself is bounded.
Since $H\subseteq c_\Zar(G)$ by our assumption, the implication (a)$\to$(c) of
Lemma \ref{lemma3}
allows us to conclude that 
$G$ contains $\Z(k)^{(\omega)}$,
where $k=exp(H)$.
Now $H$ is a \pvN\ of $G$ by
Theorem \ref{corollary2}. 

\smallskip
{\em Case 2\/}. {\sl $H$ is not bounded.\/}
In this case, we apply Theorem \ref{CH:theorem} to find 
a \map \ group topology $\tau$ on $H$. Extend $\tau$ to a Hausdorff group topology
$\tau^*$ on $G$ by taking as a base of $\tau^*$ all translates $g+U$, where $g\in G$ and $U$ is a non-empty $\tau$-open subset of $H$. Since $H$ is $\tau^*$-open and $(H,\tau)$ is \map, 
Lemma \ref{very:easy:lemma} implies $H = n(G,\tau^*)$. 
Therefore, $H$ is a \pvN\ of $G$.
\end{proof}

Theorem \ref{Main:Conjecture} allows us to answer also the following 
questions of Gabriyelyan that were open.
\begin{itemize}
\item[(a)]  \cite[Question 4.2]{G3} If $H$ is a \pvN\ of an abelian group $G$, are then all subgroups of $H$ still \pvN s of $G$? In particular, if an unbounded 
abelian group $G$ admits a \map\ group topology, is then every subgroup of $G$ a \pvN?
\item[(b)] \cite[Question 4.3]{G3} (also \cite[Problem 4]{G0}) 
Describe all infinite abelian groups $G$ such that every subgroup of $G$ is a \pvN. 
\end{itemize}

Indeed, the affirmative answer to (a) immediately follows from Theorem \ref{Main:Conjecture}. This implies also an immediate answer of (b), 
these are precisely the groups that admit a \map\ group topology. 

The paper is organized as follows. In Section \ref{HM:construction} we recall some of the properties of the The \HM\ functorial construction of a pathwise connected and locally pathwise connected group $\HMi(G)$ depending on an arbitrary topological (abelian) group $G$. 
In Section \ref{Sec:6} we collect necessary background on \map\ groups.
In  Section \ref{proofs0} we 
provide
the proof of Theorem \ref{new:corollary1}. 
In Section \ref{extension:section} we consider extension of monomorphisms into $\HMi(\T)$.
 In Section \ref{Dense:emd:HM(T)} we show that certain unbounded abelian groups (the Pr\"ufer groups and infinite direct sums of cyclic groups) admit a dense embedding in the group $\HMi(\T)$. 
Using these embeddings and the fact that the group $\HMi(\T)$ is minimal almost periodic (Corollary \ref{HM(T)isMinAP}), we 
resolve the countable case of Protasov-Comfort's problem in 
Section \ref{Sec:10}  and the general case 
 in  Section \ref{proofs}.

\section{The \HM\ construction}
\label{HM:construction}
\label{Sec:5}

Let $G$ be a group, and let $I$ be the unit interval $[0,1]$. As usual, $G^I$ denotes the set of all functions from $I$ to $G$. Clearly, $G^I$  is a group under coordinate-wise operations. For $g\in G$ and $t\in (0,1]$ let $g_t\in G^I$ be the function defined by
$$
g_t(x)=
\left\{
\begin{array}{rl}
g & \mbox{if } x< t \\
e & \mbox{if } x\ge  t,
    \end{array}
\right.
$$
where $e$ is the identity element of $G$. Note that $G_t=\{g_t:g\in G\}$ is a subgroup of $G^I$ isomorphic to $G$ for each $t\in (0,1]$. Therefore, $\HMi(G)=\sum_{t\in (0,1]} G_t$ is a subgroup of $G^I$. It is straightforward to check that this sum is direct, so that
\begin{equation}
\label{direct:decomposition:of:HM(G)}
\HMi(G)=\bigoplus_{t\in (0,1]} G_t.
\end{equation}

When $G$ is a topological group, Hartman and Mycielski \cite{HM} equip $\HMi(G)$ with a topology making it pathwise connected and locally pathwise connected. Let $\mu$ be the standard probability measure on $I$. The {\em \HM\ topology\/} on the group $\HMi(G)$ is the topology generated by taking the family of all sets of the form
\begin{equation}
\label{basic:nghb:in:HM}
O(U,\varepsilon)=\{g\in G^I: \mu(\{t\in I: g(t)\not\in U\})<\varepsilon,
\end{equation}
where $U$ is an open neighbourhood $U$ of the identity $e$ in $G$ and $\varepsilon>0$, as the base at the identity function of $\HMi(G)$.

The next lemma lists three properties of the functor $G\mapsto \HMi(G)$ that are needed for our proofs.

\begin{lemma}
\label{Dima}
\begin{itemize}
\item[(i)] For every group $G$, the group $\HMi(G)$ is pathwise connected and locally pathwise connected.{HM}
\item[(ii)] If the group $G$ is divisible and abelian, then so is $\HMi(G)$.
\item[(iii)] If the group $G$ is first countable, then so is $\HMi(G)$. 
\end{itemize}
\end{lemma}
\begin{proof}
(i) is proved in \cite{HM}; see also 
\cite{DS_ConnectedMarkov}.

(ii) Since $G_t\cong G$ for each $t\in (0,1]$, from \eqref{direct:decomposition:of:HM(G)} we conclude that  $\HMi(G)\cong G^{(\cont)}$. It follows from this that $\HMi(G)$ is divisible and abelian whenever $G$ is. 

(iii) If $G$ is first countable, then the definition of the topology of $\HMi(G)$ easily implies that $\HMi(G)$ is first countable as well.
\end{proof}

In order to investigate the fine structure of the topological group $\HMi(G)$, we need an additional notation. Let $D$ be a subgroup of $G$. For each $t\in(0,1]$,  $D_t=\{d_t:d\in D\}$ is a subgroup of $G_t$. Recalling \eqref{direct:decomposition:of:HM(G)},  we conclude that, for each non-empty subset $S$ of $(0,1]$, the sum
\begin{equation}
\label{decomposition:H(D,S)}
H(D,S)=\bigoplus_{t\in S} D_t
\end{equation}
is direct. Note that $\HMi(G)=H(G,(0,1])$.

\begin{lemma}\label{HM:groups} 
Let $G$ be a topological group.
\begin{itemize}
\item[(i)] If $D$ is a subgroup of $G$ and $S$ is a non-empty subset of $(0,1]$, then the subgroup $H(D,S)$ is
algebraically
 isomorphic to $D^{(S)}$; in particular, if $S$ is countable, then $H(D,S)\cong D^{(\omega)}$.

\item[(ii)] If $D$ is a dense subgroup of $G$ and $S$ is a dense subset of $(0,1]$, then the subgroup $H(D,S)$ of $\HMi(G)$ is dense in $\HMi(G)$.

\item[(iii)] If $D$ is a dense subgroup of $G$, then $\HMi(D)$ is a dense subgroup of $\HMi(G)$. 
\end{itemize}
\end{lemma}

\begin{proof}  (i) Since $D_t\cong D$ for each $t\in (0,1]$, from \eqref{decomposition:H(D,S)} we conclude that the subgroup $H(D,S)$ of $\HMi(G)$ is isomorphic to 
$D^{(S)}$.

To prove (ii), 
take an arbitrary element $x$ of $\HMi(G)$ and its open neighbourhood $U$. Fix a  canonical representation $x=\sum_{i=1}^n x_i$, where $x_i =(g_i)_{t_i} \in G_{t_i}$ for some $g_i\in G$ and $t_i \in [0,1]$. 

First, we use continuity of the group operation to choose open neighborhoods $V_i$ of each $x_i$ of the form $V_i= x_i + O(W_i,\varepsilon_i)$
such that $V_1 +\ldots + V_n\subseteq  U$.  

Second, we use density of $S$ in $[0,1]$ to pick $s_i\in (t_i-\varepsilon_i,t_i+\varepsilon_i)\cap S$ for all $i=1,\dots,n$. 

Third, we use density of $D$ in $G$ to find  $d_i\in (g_i+W_i)\cap D$ for all $i=1,\dots,n$.

Finally, we define $h_i:=(d_i)_{s_i}\in D_{s_i}$ for all $i=1,\dots,n$. 
Since $s_1,s_2,\dots,s_n\in S$, it follows from \eqref{decomposition:H(D,S)} that 
$h=\sum_{i=1}^n h_i\in H(D,S)$. Since
$h_i - x_i \in  O(W_i,\varepsilon_i)$, we have $h_i\in V_i$ ($i=1,\dots,n$). Thus, 
$h\in \sum_{i=1}^n V_i\subseteq U$.
Therefore,  $h\in H(D,S)\cap U\not=\emptyset$.

Since $\HMi(D)=H(D,(0,1])$, (iii) follows from item (ii), in which one takes $S=(0,1]$.
\end{proof}

\begin{corollary}\label{corollary:Moscow2}
If $D$ is a dense subgroup of a topological group $G$, the group $\HMi(G)$ contains a dense subgroup isomorphic to $D^{(\omega)}$. In particular,  
\begin{itemize}
\item[(a)] $\HMi(G)$ is separable, whenever $G$ is separable. 
\item[(b)] $\HMi(G)$ is second countable, whenever $G$ is second countable. 
\end{itemize}
\end{corollary}

\begin{proof} Let $S$ be an arbitrary countable dense subset of $(0,1]$. It follows from item (i) of Lemma \ref{HM:groups}
that $H(G,S)$ is isomorphic to  $G^{(\omega)}$. Finally, the subgroup $H(G,S)$ of $\HMi(G)$ is dense in $\HMi(G)$ by item (ii) of Lemma \ref{HM:groups}.

(a) follows directly from the first assertion of the corollary. 

(b) If $G$ is second countable, then $G$ is first countable and separable. By Lemma \ref{Dima}(iii), $\HMi(G)$ is first countable. Being a topological group,  it is metrizable. On the other hand, (a) implies that $G$ is separable. Therefore, $\HMi(G)$ is second countable as well. 
\end{proof}

\begin{corollary}
\label{HM(T):has:countable:base}
$ \HMi(\T)$ has a countable base.
\end{corollary}

\section{Background on \map\ groups}
\label{Sec:6}

In the next lemma we collect some permanence properties of the minimally almost periodic groups. 

\begin{lemma} \label{easy:lemma}
The property of being minimally almost periodic is preserved by taking direct products, direct sums, dense subgroups and completions. In particular,  if each  $G_i$ admits a minimally almost periodic group topology, then also $\bigoplus_{i\in I} G_i$ and $\prod_{i\in I} G_i$ have the same property.
\end{lemma} 

The next immediate corollary will be our main tool for producing \map \ group topologies on the abelian groups. 

\begin{corollary}\label{easy:corollary} If an abelian group $G$ can be densely embedded into a \map \ group, then $G$ admits a \map \ group topology. 
 \end{corollary} 

 The next lemma shows that under a (mild) algebraic restraint on the group $G$ all connected group topologies of $G$ are minimally almost periodic. 

\begin{lemma}\label{lemma1} Let $G$ be a connected abelian group with $r(G)<\cont$. Then $G$ is minimally almost periodic. 
 \end{lemma} 

\begin{proof} Let $\chi: G\to T$ be a character of  $G$. Then $r(\chi(G)) \le r(G) < \cont = r(\T)$. Hence, $\chi(G) \ne \T$.  On the other hand, $\chi(G)$ is a connected subgroup of $\T$, so either $\chi(G) = \T$ or $\chi(G) = \{0\}$. This proves that $\chi(G) = \{0\}$; that is, $\chi = 0$. Therefore, $G$ is  minimally almost periodic. 
\end{proof}

For every abelian group $G$, the group $\HMi(G)$ is known to be \map; see \cite{DW}. We shall only need this fact for $G=\T$, as well as for torsion groups $G$.
This is why we decided to offer a short alternative proof of minimal almost periodicity of  $\HMi(G)$ in these two special cases.

\begin{proposition}
\label{proposition:Matsuyama1}
For every topological abelian group $G$ having a dense torsion subgroup $t(G)$, the group $\HMi(G)$ is  \map.  
\end{proposition}

\begin{proof} Clearly, the group $\HMi(t(G))$ is torsion. It is abelian by Lemma \ref{Dima}(ii). Moreover, $\HMi(t(G))$ is connected by Lemma \ref{Dima}(i). Applying Lemma \ref{lemma1}, we conclude that $\HMi(t(G))$ is \map.

By Lemma \ref{HM:groups}(iii), the subgroup $\HMi(t(G))$ of $\HMi(G)$ is dense in $\HMi(G)$. Now Lemma \ref{easy:lemma} implies that  $\HMi(G)$ is \map \ too. 
\end{proof}

\begin{corollary}\label{HM(T)isMinAP}
$\HMi(\T)$ is \map.
\end{corollary}

\section{The bounded case: Proof of Theorem \ref{new:corollary1}}
\label{proofs0}

For the sake of this section only, let us call a group $N$ {\em nice\/} if it has the form $N = \bigoplus_{i=1}^n \Z(p^i)^{(\alpha_i)}$, where $p$ is a prime number, $n\in\N^+$,  the cardinal $\alpha_n$ is infinite, while all cardinals $\alpha_1,\dots,a_{n-1}$ are finite (possibly zero).

\begin{lemma}
\label{map:topologies:on:nice:groups}
Every nice group admits a \map\ group topology.
\end{lemma}

\begin{proof} Let $N$ be a nice group as defined above. Then $N=F\oplus G$, where $F=\bigoplus_{i=1}^{n-1} \Z(p^i)^{(\alpha_i)}$ is a finite (possibly zero) group and $G=\Z(p^n)^{(\alpha_n)}$. 

Fix a countable dense subset $S$ of $(0,1]$. The subgroup $H(G,S)$ of $\HMi(G)$ is dense in $\HMi(G)$ by Lemma \ref{HM:groups} (ii). By item (i) of the same lemma, $H(G,S)\cong G^{(\omega)}$. Since $\alpha_n$ is infinite, $G\cong G^{(\omega)}$. This allows us to fix an isomorphism $\xi:G\to H(G,S)$.

Let $T=(0,1]\setminus S$. Then $\HMi(G)=H(G,(0,1])=H(G,T)\oplus H(G,S)$. Since  $H(G,T)\cong G^{(T)}\cong G^{(\cont)}\cong \Z(p^n)^{(\cont)}$, our assumption on $F$ allows us to fix a monomorphism $\eta:F\to H(G,T)$.

Clearly, the sum $\theta=\eta\oplus \xi: N=F\oplus G\to H(G,T)\oplus H(G,S)=\HMi(G)$ of the monomorphisms $\eta$ and $\xi$ is a monomorphism. In particular, $N\cong \theta(N)$.

Since $\xi(G)=H(G,S)$ is dense in $\HMi(G)$, so is  $\theta(N)$.  The group $\HMi(G)$ is \map\ by Proposition \ref{proposition:Matsuyama1}. Now it remains to apply Corollary \ref{easy:corollary}. 
\end{proof}

From Lemmas \ref{easy:lemma} and \ref{map:topologies:on:nice:groups} one immediately obtains the following corollaries

\begin{corollary}
\label{direct:sums:of:nice:are;map}
A direct sum of nice groups admits a \map\ group topology.
\end{corollary}

\begin{corollary} \label{new:corollary}
A bounded abelian group admits a minimally almost periodic group topology whenever all its Ulm-Kaplanskly invariants are infinite.
\end{corollary} 

\begin{lemma}
\label{nice:groups}
A bounded $p$-group having infinite leading Ulm-Kaplansky invariant is a finite direct sum of nice groups.
\end{lemma}

\begin{proof} Straightforward  induction on the number of infinite Ulm-Kaplansky invariants of the group.
\end{proof}

\bigskip 
\noindent {\bf Proof of Theorem \ref{new:corollary1}.}
We have to prove that a bounded abelian group admits a minimally almost periodic group topology if and only if all its leading Ulm-Kaplanskly invariants are infinite. Since the necessity is clear
from Proposition \ref{Necessary:condition},
 we are left with the proof of the sufficiency.

Let $G$ be a bounded abelian group having all its leading Ulm-Kaplanskly invariants infinite. Since $G$ is a finite sum of $p$-groups, our assumption combined with Lemma \ref{nice:groups} allows us to claim that $G$ is a (finite) direct sum of nice groups. The conclusion now follows from Lemma \ref{direct:sums:of:nice:are;map}.
\qed

\section{Extention of monomorphisms into powers of $\HMi(\T)$}
\label{extension:section}

\begin{lemma}\label{embedding:lemma0}
Let $G$ be an abelian group and let $K$ be a divisible abelian group with  $r_p(K) \ge |G|$ for every prime $p$ and $p=0$. If $H$ is a  subgroup of $G$ and $j: H \to K$ is a monomorphism such that there exists a subgroup $K_1$ of $K$ with $K_1\cap j(H) = \{0\}$ and $K_1 \cong K$, then there exists a monomorphism $j': G \to K$ extending $j$. 
 \end{lemma}

\begin{proof} Since $K_1$ divisible  and $K_1\cap j(H) = \{0\}$, we can write $ K = K_0\oplus K_1$, where the subgroup $K_0$ of $K$ contains $j(H)$. 
Since, $K_1 \cong K$, $r_p(K_1) \ge |G|\ge |G/H|$ for every prime $p$ and $p=0$, so there exists a  monomorphism $m: G/H \to K$. Extend $j: H \to K_0$ to a homomorphism $j_0: G\to K_0$ and   define  $j_1: G \to K$ to be the composition of $m$ and the canonical homomorphism $ G \to G/H$. Now define  $j': G \to K = K_0\oplus K_1$ by  $j'(g) = (j_0(g), j_1(g))\in K$. If $j'(g) = $ for some $g\in G$, then $j_0(g) =  0$ and $j_1(g) = 0$. The latter equality gives $g\in H$. Hence,  $0= j_0(g) = j(g)$. Thus $g = 0$, as $j$ is a monomorphism. 
\end{proof}

\begin{lemma}\label{New:claim} Let $\kappa$ be an  infinite cardinal, $G$ be an abelian group with $|G|\leq \kappa$ and $H$ be a subgroup of $G$. 
Then every monomorphism $j:H \to \HMi(\T)^\kappa$ can be extended to a monomorphism $j': G \to \HMi(\T)^\kappa$. 
\end{lemma}

\begin{proof} From (\ref{direct:decomposition:of:HM(G)}) one deduces the algebraic isomorphisms
\begin{equation}\label{HM}
\HMi(\T) \cong \T^{(\cont)}\cong  (\Q\oplus \Q/\Z)^{(\cont)}.
\end{equation}
Hence, $K = \HMi(\T)^\kappa$ is a divisible group with 
\begin{equation}\label{ranks}
r_p(K) = 2^\kappa >  |G| \ge r_p(G)\ \mbox{ for every  }\ p \in \{0\}\cup \Prm. 
\end{equation}
 Since $|j(H)|\leq \kappa$,  its divisible hull $D$ in $K$ satisfies $|D|\leq \kappa$. We can split $K = D \oplus D'$, where $D' \cong K/D$ is divisible, so  completely determined by its $p$-ranks that are the same those of $K$, in view of (\ref{ranks}) and $|D|\leq \kappa$. Hence, $D'\cong \HMi(\T)$, 
so Lemma \ref{embedding:lemma0} applies.

The extension of a monomorphism $j:H \to K$ to a monomorphism $j': G \to K$ can be obtained also by a direct application of \cite[Lemma 3.17]{DS_Forcing}. 
 \end{proof}

From (\ref{HM}) one deduces the algebraic isomorphism $\HMi(\T) \cong \HMi(\T)^\omega$. Hence, with $\kappa = \omega$ Lemma \ref{New:claim} gives:

\begin{corollary}\label{embedding:lemma}
Let $G$ be a countable abelian group. If $H$ is a subgroup of $G$ and $j: H \to \HMi(\T)$ is a monomorphism, then there exists a monomorphism $j': G \to \HMi(\T)$ extending $j$.  \end{corollary}

\section{Countable groups admitting dense embeddings in $\HMi(\T)$}\label{Dense:emd:HM(T)}

The following lemma is a particular case of \cite[Lemma 4.1]{DS_HMP}.
\begin{lemma}
\label{simple:lemma}
Given a real number $\delta>0$, one can find  $j\in\N$ such that for every integer $k\ge j$, each element $x\in\T$ and every 
open arc $A$ in $\T$ of length $\delta$, there exists $y\in A\setminus\{0\}$ such that  $ky=x$.
\end{lemma}

\begin{lemma}
\label{roots}
Given  a non-empty open subset $W$ of $\HMi(\T)$, one can find an integer $j\in\N$ such that, for every $g\in \HMi(\T)$ and each integer $k\ge j$, there exists $h\in W$  satisfying $kh=g$.
\end{lemma}

\begin{proof} Since $W$ is non-empty, we can fix $w\in W$. Since $W$ is open in $\HMi(\T)$, there exist $\varepsilon>0$ and an open neighbourhood $U$ of $0$ in $\T$ such that $w+O(U,\varepsilon)\subseteq W$. Obviously, $U$ contains an open arc $L$ of length $\delta$ for sufficiently small $\delta$.
For this $\delta$,  we apply Lemma \ref{simple:lemma} to find $j\in\N$ satisfying the conclusion of this lemma.

Let $g\in \HMi(\T)$. Since $g,w\in \HMi(\T)$, there exist $m\in\N$, $t_0, t_1,t_2,\dots,t_m\in\R$,
$x_1,x_2,\dots,x_m\in\T$  and $z_1,z_2,\dots,z_m\in \T$
such that 
$0=t_0<t_1<t_2<\dots<t_{m-1}<t_m\le 1$, $g(s)=x_l$ 
and  $w(s)=z_l$ whenever $1\le l\le m$ and $t_{l-1}\le s<t_l$.

Let $k\ge j$ be an integer. Since $j$ satisfies the conclusion of Lemma \ref{simple:lemma}, for every integer $l$ with $1\le l\le m$, we can find an element $y_l$ inside the arc $z_l+L$ such that $k y_l = x_l$.

Finally, define $h\in \HMi(\T)$ by letting $h(s)=y_l$ whenever $1\le l\le m$ and $t_{l-1}\le s<t_l$. Clearly, $k h = g$ by our construction.
Furthermore, since $y_l\in z_l+L\subseteq z_l+U$ for every $l=1,2,\dots,m$, it follows that $h\in w+O(U,\varepsilon)\subseteq W$.
\end{proof}

\begin{lemma}\label{MoscowLemma1}
For every prime number $p$, the Pr\"ufer group $\Z(p^\infty)$ is algebraically isomorphic to a dense subgroup of $\HMi(\T)$.
\end{lemma}

\begin{proof} By Corollary \ref{HM(T):has:countable:base},  $\HMi(\T)$ has a countable base. Let $\mathscr{B}=\{V_n:n\in\N^+\}$ be an enumeration of some countable base for $\HMi(\T)$ such that all $V_n$ are non-empty.

Let $g_0\in \HMi(\T)$ be an arbitrary element of order $p$. By induction on $n\in\N^+$ we shall define an element $g_n\in \HMi(\T)$ and an integer $i_n\in\N^+$  satisfying the following conditions:
\begin{itemize}
\item[(i$_n$)] $g_n\in V_n$;
\item[(ii$_n$)]
$p^{i_n} g_n=g_{n-1}$.
\end{itemize}
Assume that $g_{n-1}$ and $i_{n-1}$ satisfying (i$_{n-1}$) and (ii$_{n-1}$) have already been constructed. Applying Lemma \ref{roots}
to $W=V_n$, we select $j$ as in the conclusion of this lemma.  Next we fix $i_n\in\N^+$ such that $p^{i_n}\ge j$. Applying Lemma \ref{roots} to 
$g=g_{n-1}$ and $k=p^{i_n}$, we can find $g_n\in \HMi(\T)$ satisfying (i$_n$) and (ii$_n$). The inductive construction is complete.

Let $G$ be the subgroup of $\HMi(\T)$ generated by the set $\{g_n:n\in\N\}$. Since (i$_n$) holds for every $n\in\N^+$ and $\mathscr{B}$ is a base for $\HMi(\T)$, it follows that $G$ is dense in $\HMi(\T)$. Since $g_0$ has order $p$, and (ii$_n$) holds for all $n\in\N$, we conclude that $G$ is isomorphic to $\Z(p^\infty)$.
\end{proof}

For every $g\in \HMi(\T)$, we let 
$$
S(g)=\{t\in[0,1]: g(t)\not=0\}.
$$

Let
$$
HM^*(\T)=\{g\in \HMi(\T): \sup S(g)<1\}.
$$

\begin{lemma}\label{MoscowLemma2} Given a real number $\eta< 1$ and a non-empty open subset $W$ of $\HMi(\T)$, one can find an integer $j\in\N$ such that,
for each integer $k\ge j$, there exists $h\in W$ satisfying the following conditions:
\begin{itemize}
  \item[(i)] $kh=0$;
  \item[(ii)] $h(s)\not=0$ for some $s\in [0,1]$ with $\eta<s$;
  \item[(iii)] $h\in HM^*(\T)$.
\end{itemize}
\end{lemma}

\begin{proof} Since $W$ is non-empty, we can fix $w\in W$. Since $W$ is open in $\HMi(\T)$, there exist $\varepsilon>0$ and an open neighbourhood $U$ of $0$ in $\T$ such that $w+O(U,\varepsilon)\subseteq W$. Obviously, $U$ contains an open arc $L$ of length $\delta$ for sufficiently small $\delta$.
For this $\delta$,  we apply Lemma \ref{simple:lemma} to find $j\in\N$ satisfying the conclusion of this lemma.

Since $w\in \HMi(\T)$, there exist $m\in\N$, $t_0, t_1,t_2,\dots,t_m$ and $z_1,z_2,\dots,z_m\in \T$ such that $0=t_0<t_1<t_2<\dots<t_{m-1}<t_m\le 1$
and $w(s)=z_l$ whenever $1\le l\le m$ and $t_{l-1}\le s<t_l$. By subdividing further if necessary, we may assume, without loss of generality, that $1-\gamma<t_{m-2}$ and $1-\varepsilon<t_{m-1}$. 

Let $k\ge j$ be an integer. Since $j$ satisfies the conclusion of Lemma \ref{simple:lemma}, for every integer $l$ with $1\le l\le m-1$, 
we can find a non-zero element $y_l$ inside the arc $z_l+L$ such that $k y_l = 0$.

Define $h\in \HMi(\T)$ by letting $h(s)=y_l$ whenever $1\le l\le m-1$ and $t_{l-1}\le s<t_l$, and letting $h(s)=0$ whenever $t_{m-1}\le s\le 1$. Since $y_l\in z_l+L\subseteq z_l+U$ for every $l=1,2,\dots,m-1$,  and the measure of the interval $[t_{m-1},1]$ is less than $\varepsilon$ by the choice of $t_{m-1}$, it follows that $h\in w+O(U,\varepsilon)\subseteq W$.

Clearly, $k h = g$ by our construction, so (i) holds. Furthermore, (ii) holds by our choice of $t_{m-2}$. Indeed, defining $s=(t_{m-2}+t_{m-1})/2$, we get $\eta<t_{m-2}<s<1$ and $h(s)=y_{m-1}\not=0$. Finally, since $h(s)=0$ whenever $t_{m-1}\le s\le 1$, one has $\sup S(h)\le t_{m-1}<1$, which implies (iii). 
\end{proof}

\begin{lemma}\label{MoscowLemma3} For every $n\in\N$, let $C_n$ be a cyclic group of order $a_n$. 
If $\lim_{n\to\infty} a_n=\infty$, then  $G=\bigoplus_{n\in\N} C_n$ is algebraically isomorphic to a dense subgroup of $\HMi(\T)$. 
\end{lemma}

\begin{proof} By Corollary \ref{HM(T):has:countable:base},  $\HMi(\T)$ has a countable base. Let $\mathscr{B}=\{V_n:n\in\N^+\}$ be an enumeration of some countable base for $\HMi(\T)$ such that all $V_n$ are non-empty.

By induction on $n\in\N^+$ we shall define an element $g_n\in \HMi(\T)$ and an integer $i_n\in\N^+$  satisfying the following conditions:
\begin{itemize}
\item[(i$_n$)] $g_n\in V_n \cap HM^*(\T) $;
\item[(ii$_n$)] $a_ng_n= 0$; 
\item[(iii$_n$)] $\langle g_n \rangle \cap \langle g_1,\ldots  g_{n-1} \rangle = \{0\}$ when $n>1$.  
\end{itemize}

With $\eta = 1/2$ and $W=V_1$ apply Lemma \ref{MoscowLemma2} to find an integer $j_1$ as in the lemma. Then choosing $i_1$ with $a_{i_1} \ge j_1$ there exists $g_1 \in V_1\cap HM^*(\T)$ with $a_{i_1} g_1= 0$ and $h(s) \ne 0$ for some $s > 1/2$.  Assume that $n>1$ and  $g_{n-1}\in V_{n-1}\cap HM^*(\T) $ with $a_{i_{n-1}} g_{n-1}= 0$ have already been constructed. Applying Lemma \ref{MoscowLemma2} to 
$$
W=V_n\mbox{ and }\eta = \max \{\sup S(g_i): i=1,2,\ldots, n-1\}< 1
$$
we choose $j_n$ as in the conclusion of this lemma. Then choosing $i_n$ with $a_{i_n} \ge j_n$ there exists $g_n \in V_1\cap HM^*(\T)$ with $a_{i_n} g_n= 0$ and $g_n(s) \ne 0$ for some $s > \eta$.  The latter property ensures that (iii$_n$) holds true, while (i$_n$) and (ii$_n$) are obviously satisfied. The inductive construction is complete.

Let $H$ be the subgroup of $\HMi(\T)$ generated by the set $\{g_n:n\in\N\}$.
Since (i$_n$) holds for every $n\in\N^+$ and $\mathscr{B}$ is a base for $\HMi(\T)$, it follows that $H$ is dense in $\HMi(\T)$.

 For every $n\in\N$ denote by  $C'_n$ the cyclic subgroup of $\HMi(\T)$ generated by $g_n$ and let $m_n = o(g_n)$ be its order. 
 Since $m_n| a_{i_n}$ by  (ii$_n$),  $C'_n$ is a cyclic group isomorphic to a subgroup of $C_n$.  As (iii$_n$) holds for all $n\in\N$, we conclude that $H$ is isomorphic to $\bigoplus_{n\in\N} C'_n$. From $C_n'\leq  C_{i_n}$ we deduce that $H$ is isomorphic to a subgroup of $G$. Let $j: H \to G$ be the monomorphism witnessing that.  
Since $\HMi(\T)$ is divisible and since $G$ is countable, while $r_p(\HMi(\T)) = \cont$, there exists a monomorphism $\nu: G \to \HMi(\T)$ that composed with the monomorphism $j$ gives the inclusion $H \hookrightarrow \HMi(\T)$ (see Lemma \ref{embedding:lemma}).  In particular, $\nu(G)$ contains the dense subgroup $H$, hence $\nu(G)$ is a dense subgroup of $\HMi(\T)$. 
\end{proof}

\begin{lemma}
\label{basic:subgroup}
An unbounded torsion abelian group $G$ contains a subgroup algebraically isomorphic to one of the following three groups:
\begin{itemize}
\item[(i)] the Pr\"ufer group $\Z(p^\infty)$, for some prime number $p$;
\item[(ii)] the direct sum $\bigoplus _{n=1}^\infty \Z(p^n)^{(\alpha_n)}$, 
where infinitely many ordinals $\alpha_n$ are non-zero;
\item[(iii)] the direct sum $\bigoplus_{p\in \pi}\Z(p)$ for a suitable infinite set $\pi$ of prime numbers.
\end{itemize}
\end{lemma}
\begin{proof}
If $G$ contains non-trivial divisible subgroups, then $G$ contains  the Pr\"ufer group $\Z(p^\infty)$ for some prime number $p$. Thus (i) holds.

Assume now that $G$ is reduced, i.e., $G$ contains no non-trivial divisible subgroups. 
For every $p\in \Prm$ let $G_p$ be the $p$-torsion subgroup of $G$ and let $\pi :=\{p\in \Prm: G_p\ne \{0\} \}$. 

If $\pi$ is infinite, then pick an element $x_p\in G_p$ with $o(x_p) =p$ for each $p\in \pi$ and let $A$ be the subgroup generated by the set $\{ x_p : p\in \pi\}$. Then $A\cong \bigoplus_{p\in \pi}\Z(p)$, so (iii) holds.

Assume now that $\pi$ is finite.
Then for some prime $p\in \pi$ the group $G_p$ must be unbounded. According to \cite{Fuchs}, there exists a basic subgroup $B$ of $G_p$, i.e., a subgroup with the following three properties: 
\begin{itemize}
  \item[(a)] $B$ is a direct sum of cyclic groups, i.e., $B = \bigoplus _{n=1}^\infty \Z(p^n)^{(\alpha_n)}$;
  \item[(b)] $B$ is pure, i.e., $p^n G_p \cap B = p^n B$ for every for some $n\in \N^+$; 
  \item[(c)] $G_p/B$ is divisible. 
\end{itemize}

It is enough to show that $B$ is not bounded, which would yield (ii).
Suppose that $B$ is bounded; that is, $p^n B = 0$ for some $n\in \N^+$, then from (b) we deduce that $p^n G_p \cap B =\{0\}$. From (c) we deduce that 
$$
(p^nG_p + B)/B=p^n(G_p/B) = G_p/B,
$$ 
hence $G_p = p^nG_p + B$. Since this sum $G_p = p^nG_p \oplus B$ is direct, (c) implies that $G_p$ contains a subgroup, namely $p^nG_p\cong G_p/B$ that is divisible group.  Since $G_p$ is reduced, we conclude that $p^n G$ is trivial, so $G_p=B$ is bounded, a contradiction. 
\end{proof}

\begin{theorem}
\label{embedding:countable:groups}
Let $G$ be a countable abelian group satisfying at least one of the following conditions:
\begin{itemize}
\item[(a)] the torsion part $t(G)$ of $G$ is unbounded;
\item[(b)] the rank of $G$ is infinite.
\end{itemize}
Then $G$ is algebraically isomorphic to a dense subgroup of $\HMi(\T)$.
\end{theorem}
\begin{proof}
It suffices to find a subgroup $H$ of $G$ which admits a dense embedding into 
$\HMi(\T)$. Indeed, suppose that $H$ is such a group, and let $j:H\to \HMi(\T)$ be a monomorphism such that $j(H)$ is dense in $\HMi(\T)$.
By Corollary \ref{embedding:lemma}, $j$ can be extended to 
a monomorphism $j': G\to \HMi(\T)$. Clearly, $j'(G)\cong G$.
Since $j(H)$ is dense in $\HMi(\T)$ and $j(H)\subseteq j'(G)\subseteq  \HMi(\T)$, the subgroup $j'(G)$ is dense in $\HMi(\T)$. 

To find a subgroup $H$ of $G$ which admits a dense embedding into 
$\HMi(\T)$, we consider two cases.

\smallskip
{\em Case 1\/}. {\sl Item (a) holds\/}. Applying Lemma \ref{basic:subgroup} to $t(G)$, we conclude that $t(G)$ (and thus, $G$ as well) contains a subgroup $H$ algebraically isomorphic to one of the three groups listed in  Lemma \ref{basic:subgroup}. In case (i) of Lemma \ref{basic:subgroup}, $H$ admits a dense embedding into $\HMi(\T)$ by Lemma \ref{MoscowLemma1}, while in cases (ii) and (iii) of Lemma \ref{basic:subgroup}, $H$ admits a dense embedding into $\HMi(\T)$ by Lemma \ref{MoscowLemma3}.

\smallskip
{\em Case 2\/}. {\sl Item (b) holds\/}. 
In this case,
$G$ contains a subgroup $H$ algebraically isomorphic to $\Z^{(\omega)}$. 
Since $\T$ contains a dense subgroup $D$ algebraically isomorphic to $\Z$,
Corollary \ref{corollary:Moscow2} implies that $\HMi(\T)$ contains a dense subgroup $N$ with $N\cong D^{(\omega)}\cong \Z^{(\omega)}\cong H$.
\end{proof}

\begin{remark}
\label{finitely:generated:are:not:dense}
It can be easily checked that {\em no finitely generated subgroup of $\HMi(\T)$ can be dense in $\HMi(\T)$\/}. 
In particular, finite direct sums 
of cyclic groups can never be dense in $\HMi(\T)$. 
\end{remark}

\section{The countable case of Comfort-Protasov problem: Nienhuys group enters briefly}
\label{Nienhyus:section}
\label{Sec:10}

In this section we give 
a short self-contained proof of the first part of Protasov-Comfort problem; namely, we show that {\em every countable unbounded abelian group admits a  \map \ group topology\/}. To achieve this, we need to recall the construction of the minimally almost periodic and monothetic group $\Ni$ of Nienhuys \cite{N}, paying special attention only to the {\em algebraic\/} structure of  $\Ni$.

Let $c_0(\T)$ be the subgroup of $\T^\omega$ consisting of all null sequences in $\T$ (equipped with the norm topology in  $\T^\omega$). 
As $c_0(\T)$ is divisible and contains $\T^{(\omega)}$, it follows that $r_0(c_0(\T)) = \cont$ and $r_p(c_0(\T))\geq \omega$ for every $p\in \Prm$. 
On the other hand, if $x = (x_n) \in c_0(\T)$ and $px =0$, then $px_n= 0$ for all $n\in \N$. Therefore, $x_n \to 0$ implies that all but finitely many $x_n$ vanish, i.e., $x\in \T^{(\omega)}$. This proves that $c_0(\T)[p] = \T^{(\omega)}[p] = \T[p]^{(\omega)}$, i.e., $r_p(c_0(\T)) = \omega$. 

The monothetic \map \ group $\Ni$, built by Nienhuys \cite{N}, is a quotient of $c_0(\T)$ with respect to a cyclic subgroup.
Hence, $\Ni$ is divisible and has the same $p$ ranks as $c_0(\T)$. Therefore, one has the algebraic isomorphisms
\begin{equation}\label{Nienhuys:equation}
\Ni \cong c_0(\T)\cong \Q^{(\cont)} \oplus (\Q/\Z)^{(\omega)}.
\end{equation}

\begin{lemma}
\label{Saak}
Every countable unbounded abelian group admits a \map\ group topology.
\end{lemma}
\begin{proof}
Let $G$ be a countable unbounded abelian group. We consider two cases.

\smallskip
{\em Case 1\/}. {\em Either item (a) or item (b) of Theorem \ref{embedding:countable:groups} holds\/}. In this case, the conclusion of Theorem \ref{embedding:countable:groups} guarantees that $G$ can be densely embedded into 
the 
group 
$\HMi(\T)$. Since the latter group is \map\ by Corollary \ref{HM(T)isMinAP},
applying Corollary \ref{easy:corollary} we conclude that $G$ admits a \map\ group topology.

\smallskip
{\em Case 2\/}. {\em Neither item (a) nor item (b) of Theorem \ref{embedding:countable:groups} holds\/}. In this case, $G$ has finite rank. Since $t(G)$ is bounded and $G$ is unbounded, $G$ contains a subgroup $H$ such that $H\cong\Z$.
Since $\Ni$ is monothetic, we can 
fix a dense embedding $j: H \to \Ni$. Since the divisible hull $D$ of $j(H)$ in $\Ni$ is isomorphic to $\Q$, it follows from (\ref{Nienhuys:equation}) that one can 
split $\Ni = D \oplus D_1$, where $D_1\cong \Ni/D\cong \Ni$ is a divisible group. Now we can apply Lemma \ref{embedding:lemma0}
to extend $j$ to a dense embedding $j': G \to \Ni$. 
Since $\Ni$ is \map, Corollary \ref{easy:corollary} implies that $G$ admits a \map\ group topology.
\end{proof}

\begin{remark}
\label{where:Nienhuys:is:used}
A careful analysis of the proof of Lemma \ref{Saak} shows that only the (countable) groups of finite non-zero rank require 
the recourse to the Nienhuys group $\Ni$. 
\end{remark}

\section{The general case: Proof of Theorem \ref{CH:theorem}}\label{proofs}

We recall here a fundamental notion from \cite{DGB}.

\begin{definition}
\label{w-divisible:reformulation}
\cite{DGB}
An abelian group $G$ is called {\em $w$-divisible\/} if $|mG|=|G|$ for all integers $m\ge 1$.
\end{definition}

\begin{lemma}
\label{w-divisible:embeddings}
Every uncountable $w$-divisible group is algebraically isomorphic to a dense 
subgroup of $\HMi(\T)^G$.
\end{lemma}
\begin{proof}
Our aim is to find a family $\{G_i:i\in I\}$ of subgroups of $G$ such that $|I|=|G|$, $G$ contains its direct sum $H=\bigoplus_{i\in I} G_i$ 
and each $G_i$ admits a dense embedding into $\HMi(\T)$. Assuming this is done, one obtains a monomorphism
$j: H =\bigoplus_{i\in I} G_i\to \bigoplus_{i\in I}\HMi(\T) = \HMi(\T)^{(I)}$ 
such that $j(H)$ is dense in $\HMi(\T)^{(I)}$. Since $\HMi(\T)^{(I)}$ is dense in $\HMi(\T)^I$,
the subgroup $j(H)$ of $\HMi(\T)^I$ is also dense in $\HMi(\T)^I$.
Using Lemma \ref{New:claim}, this embedding can be extended to a monomorphism $j': G \to \HMi(\T)^{I}$, so that $j'(G)\cong G$ and $j'(G)$ is also dense in $\HMi(\T)^{I}$. Since $|I|=|G|$, the topological groups $\HMi(\T)^{I}$ and $\HMi(\T)^{G}$ are topologically isomorphic. This produces the desired 
dense embedding of $G$ into $\HMi(\T)^G$.

According to 
\cite[Theorem 3.6]{DS_ConnectedMarkov}, 
$G$ contains a direct sum $\bigoplus_{s\in S}H_s$ of unbounded groups with $|S| = |G|$. It is not restrictive to assume that, for each $s\in S$, either $H_s\cong \Z$ or $H_s$ is a countable unbounded torsion group. Let
$S_0 = \{s\in S: H_s \cong \Z\}$ and $S_1 = S \setminus S_0$.
Since $S$ is infinite, there exists $k = 0,1$ such that $|S|=|S_k|$. 

\smallskip
{\em Case 1\/}. $k=0$.
Since $S_0$ is infinite (in fact, even uncountable), we can find a decomposition $S_0=\bigcup_{i\in I} T_i$ of $S$ into countably infinite pairwise disjoint sets $T_i$. For each $i\in I$, we let $G_i=\bigoplus_{t\in T_i} H_t$.
Since $H_t\cong \Z$ for every $t\in T_i$, it follows that 
$G_i\cong \Z^{(\omega)}$ for every $i\in I$. Furthermore,
$G$ contains direct sum 
$\bigoplus_{s\in S_0} H_s=\bigoplus_{i\in I}\bigoplus_{t\in T_i} H_t
=
\bigoplus_{i\in I}G_i$.

\smallskip
{\em Case 2\/}. $k=1$.
In this case we let $I=S_1$ and $G_i=H_i$ for all $i\in I=S_1$.

\smallskip
We claim that $\{G_i: i\in I\}$ is the desired family.
Indeed, $|I|=|G|$. By our construction, $G$ contains the direct sum 
$H=\bigoplus_{i\in I} G_i$.
Finally, each $G_i$ is either a countable unbounded torsion group or the group $\Z^{(\omega)}$ of infinite rank. Applying Theorem \ref{embedding:countable:groups}, we conclude that $G_i$ admits a dense embedding into $\HMi(\T)$.
\end{proof}

\begin{corollary}\label{new:w-divisible:theorem}
Every $w$-divisible group admits a \map \  group topology. 
\end{corollary}

\begin{proof} Let $G$ be a $w$-divisible group. If $G$ is 
a bounded torsion group, then $G$ must be trivial 
by Definition \ref{w-divisible:reformulation}. Clearly, the trivial group has a \map\ group topology. Therefore, from now on we shall assume that $G$ is not bounded torsion. 

If $G$ is countable, then $G$ admits a \map\ group topology by Lemma \ref{Saak}. 

If $G$ is uncountable,
then $G$ admits a dense embedding into $\HMi(\T)^G$ by Lemma \ref{w-divisible:embeddings}.
Since $\HMi(\T)$ is \map\ by Corollary \ref{HM(T)isMinAP}, its power
$\HMi(\T)^G$ is \map\ by Lemma \ref{easy:lemma}.
Now $G$ admits a \map\ group topology by Corollary \ref{easy:corollary}.
\end{proof}

We need the following lemma that can be obtained from \cite[Lemma 4.5]{DS_ConnectedMarkov} with $\sigma = \omega$: 

\begin{lemma}
\label{homogeneous:split} \cite{DS_ConnectedMarkov}
Every unbounded abelian group $G$ admits a decomposition $G=N\oplus H$ such that $N$ is a bounded group with all its Ulm-Kaplanskly invariants infinite and  $H$ is a $w$-divisible group.
\end{lemma}

\medskip 
\noindent {\bf Proof of Theorem \ref{CH:theorem}.}
 Let $G$ be an unbounded abelian group. By Lemma \ref{homogeneous:split}, we find a decomposition $G=N\oplus H$ such that  $N$ is either trivial, or a bounded group with infinite Ulm-Kaplanskly invariants and $H$ is a 
$w$-divisible group. 
Applying Corollary \ref{new:corollary}, we conclude that $N$ admits a minimally almost periodic group topology
$\mathscr{T}_N$. 
By Corollary \ref{new:w-divisible:theorem}, $H$ also admits a minimally almost periodic group topology $\mathscr{T}_H$.
Since both $(N,\mathscr{T}_N)$ and $(H,\mathscr{T}_H)$ are minimally almost periodic, so is their product $(N,\mathscr{T}_N)\times(H,\mathscr{T}_H)$, by Lemma \ref{easy:lemma}. 
\qed

\begin{remark}
A careful analysis of the proofs in this section shows that our use of the Nienhuys group is restricted to the reference to Lemma \ref{Saak} in the proof 
of Lemma \ref{new:w-divisible:theorem}. 
Combining this with Remark \ref{where:Nienhuys:is:used}, we conclude that 
the recourse to Nienhuys group $\Ni$ in the proof of our main results is necessary only for handling countable groups of finite non-zero rank.
\end{remark}

\end{document}